\documentclass[a4paper, 1wpt]{amsart}
\usepackage{amsmath,amsthm,amssymb,amsfonts,amscd,color}
\usepackage{pstricks}
\usepackage{pst-node}

\theoremstyle{plain} \numberwithin{equation}{section}
\newtheorem{theorem}{Theorem}[section]
\newtheorem{corollary}[theorem]{Corollary}
\newtheorem{proposition}[theorem]{Proposition}
\newtheorem{lemma}[theorem]{Lemma}

\theoremstyle{definition}
\newtheorem*{definition}{Definition}
\newtheorem{example}{Example}[section]
\newtheorem*{remark}{Remark}

\def\Z{\mathbb Z}
\def\C{\mathbb C}
\def\R{\mathbb R}

\def\dia{\diamondsuit}

\DeclareMathOperator{\Hom}{Hom}

\DeclareMathOperator{\Cone}{Cone}

\DeclareMathOperator{\im}{im}

\newcommand{\cP}{\mathcal{P}}
\newcommand{\cF}{\mathcal{F}}
\newcommand{\nbd}{\mathrm{nbd }}

\def\red#1{\textcolor{red}{#1}}
\def\blue#1{\textcolor{blue}{#1}}

\begin{document}
\title{Toric origami manifolds and multi-fans}
\author[M. Masuda]{Mikiya Masuda}
\address{Department of Mathematics, Osaka City University, Sumiyoshi-ku, Osaka 558-8585, Japan.}
\email{masuda@sci.osaka-cu.ac.jp}

\author[S. Park]{Seonjeong Park}
\address{Division of Mathematical Models, National Institute for Mathematical Sciences, 463-1 Jeonmin-dong, Yuseong-gu, Daejeon 305-811, Korea}
\email{seonjeong1124@nims.re.kr}

\dedicatory{Dedicated to Professor Victor M. Buchstaber  on his 70th birthday}

\date{\today}
\thanks{The first author was partially supported by Grant-in-Aid for Scientific Research 25400095}

\begin{abstract}
   The notion of a toric origami manifold, which weakens the notion of a symplectic toric manifold, was introduced by Cannas da Silva-Guillemin-Pires \cite{ca-gu-pi11} and they show that toric origami manifolds bijectively correspond to origami templates via moment maps, where an origami template is a collection of Delzant polytopes with some folding data.  Like a fan is associated to a Delzant polytope, a multi-fan introduced in \cite{ha-ma03} and \cite{masu99} can be associated to an oriented origami template.  In this paper, we discuss their relationship and show that any simply connected compact smooth $4$-manifold with a smooth action of $T^2$ can be a toric origami manifold. We also characterize products of even dimensional spheres which can be toric origami manifolds.
\end{abstract}

\maketitle

\section*{Introduction}

A \emph{symplectic toric} (or \emph{toric symplectic}) manifold is a compact connected symplectic manifold with an effective Hamiltonian action of a torus $T$ with $\dim T=\frac{1}{2}\dim M$.   Any projective smooth toric variety with the restricted compact torus action is a symplectic toric manifold.  A famous theorem by Delzant \cite{delz88} says that the correspondence from symplectic toric manifolds to simple convex polytopes called \emph{Delzant polytopes} (see \cite{guil94}) via moment maps is bijective. The normal fan of a Delzant polytope is complete and non-singular. Therefore, the theorem of Delzant implies that equivariant diffeomorphism types of symplectic toric manifolds are exactly same as those of projective smooth toric varieties.

A \emph{folded symplectic form} on a $2n$-dimensional manifold $M$ is a closed $2$-form $\omega$ whose top power $\omega^n$ vanishes transversally on a subset $Z$ and whose restriction to points in $Z$ has maximal rank. Then $Z$ is a codimension-one submanifold of $M$ and called the \emph{fold}.  Cannas da Silva \cite{cann10} shows using the h-principle that any orientable compact smooth manifold $M$ of even dimension admits a folded symplectic form if and only if $M$ admits a stably almost complex structure.  As a corollary, it follows that any orientable compact smooth $4$-manifold admits a folded symplectic form.

The maximality of the restriction of $\omega$ to $Z$ implies the existence of a line field on $Z$ and $\omega$ is called an \emph{origami form} if the line field is the vertical bundle of some principal $S^1$-fibration $Z\to B$.  The notions of a Hamiltonian action and a moment map can be defined  similarly to the symplectic case and a \emph{toric origami manifold} is defined to be a compact connected origami manifold $(M,\omega)$ equipped with an effective Hamiltonian action of a torus $T$ with $\dim T=\frac{1}{2}\dim M$.

Cannas da Silva, Guillemin and Pires \cite{ca-gu-pi11} show that toric origami manifolds bijectively correspond to origami templates via moment maps, where an origami template is a collection of Delzant polytopes with some folding data.  This result is a generalization of the theorem of Delzant to toric origami manifolds.  A toric origami manifold is not necessarily orientable and oriented toric origami manifolds correspond to \emph{oriented} origami templates.  Like a fan is associated to a Delzant polytope, a multi-fan introduced in \cite{ha-ma03} and \cite{masu99} can be associated to an oriented origami template.  A \emph{multi-fan} is a collection of cones satisfying certain conditions, where cones may overlap unlike an ordinary fan.

In this paper, we discuss how to associate a multi-fan to an oriented origami template and their relationship. Then we show that any simply connected closed smooth $4$-manifold with a smooth action of $T^2$ can be a toric origami manifold.
The product of a toric origami manifold and a symplectic toric manifold is a toric origami manifold with the product form.  However, the product of two origami forms with nonempty folds is not an origami form.  Therefore, it is natural to ask whether the product of two toric origami manifolds admits a toric origami form.  We discuss this problem and characterize products of even dimensional spheres which admit toric origami forms.

This paper is organized as follows. We recall the definitions and properties of multi-fans and torus manifolds in Section~\ref{sec:2} and those of toric origami manifolds and origami templates in Section~\ref{sec:3}.  In Section~\ref{sec:4}, we discuss how to associate a multi-fan to an origami template. In Section~\ref{sec:4-dimensional case}, we show that any simply connected compact smooth $4$-manifold with a smooth action of $T^2$ is equivariantly diffeomorphic to a toric origami manifold.  In Section~\ref{sec:6}, we discuss whether the product of toric origami manifolds admits a toric origami form and characterize products of even dimensional spheres which admit toric origami forms.

\section{Multi-fans and torus manifolds} \label{sec:2}

We recall the definitions and properties of multi-fans and torus manifolds, and then see their relationship.  Details can be found in \cite{ha-ma03}.

Let $N$ be a lattice of rank $n$, which is isomorphic to $\Z^n$. We denote the real vector space $N\otimes\R$ by $N_\R$. A cone $\sigma$ in $N$ means a strongly convex rational polyhedral cone (with apex at the origin), that is, there exists a finite number of vectors $v_1,\dots,v_m$ in $N$ such that
\[
\sigma=\{r_1v_1+\dots+r_mv_m \mid r_i\in\R \text{ and } r_i\ge 0\text{ for all $i$}\}
\quad \text{and}\quad  \sigma\cap(-\sigma)=\{0\}.
\]
A cone is called \emph{simplicial} if it is generated by linearly independent vectors.  If the generating vectors can be taken as a part of a basis of $N$, then the cone is called \emph{non-singular}.

Denote by $\Cone(N)$ the set of all cones in $N$. The set $\Cone(N)$ has a (strict) partial ordering $\prec$ defined by: $\tau\prec\sigma$ if and only if $\tau$ is a proper face of $\sigma$. The cone $\{0\}$ consisting of the origin is the unique minimum element in $\Cone(N)$. Let $\Sigma$ be a partially ordered finite set with a unique minimum element, e.g. $\Sigma$ is a finite simplicial complex with an empty set added, where the partial ordering is the inclusion relation. We denote the (strict) partial ordering on $\Sigma$ by $<$ and the minimum element by $\ast$. Suppose that there is a map $$C\colon \Sigma\to \Cone(N)$$ such that
\begin{enumerate}
\item $C(\ast)=\{0\}$;
\item If $I<J$ for $I,J\in\Sigma$, then $C(I)\prec C(J)$;
\item For any $J\in\Sigma$ the map $C$ restricted on $\{I\in\Sigma\mid I\leq J\}$ is an isomorphism of ordered sets onto $\{\kappa\in\Cone(N)\mid \kappa\preceq C(J)\}$.
\end{enumerate}
For an integer $m$ such that $0\leq m\leq n$, we set $$\Sigma^{(m)}:=\{I\in\Sigma\mid\dim C(I)=m\}.$$ Now, we consider two functions $$w^\pm\colon\Sigma^{(n)}\to\Z_{\geq0}$$ such that $w^+(I)>0$ or $w^-(I)>0$ for every $I\in\Sigma^{(n)}$.
\begin{definition}
A triple $\Delta:=(\Sigma,C,w^\pm)$ is a \emph{multi-fan} in $N$ and the dimension of $\Delta$ is defined to be the rank of $N$. The multi-fan $\Delta:=(\Sigma,C,w^\pm)$ is called \emph{simplicial} (resp. \emph{non-singular}) if every cone in $C(\Sigma)$ is simplicial (resp. non-singular).
\end{definition}
A vector $v\in N_\R$ is called \emph{generic} if $v$ does not lie on any linear subspace spanned by a cone in $C(\Sigma)$ of dimension less than $n$. For a generic vector $v$ we set $d_v=\sum_{v\in C(I)}(w^+(I)-w^-(I))$, where the sum is understood to be zero if there is no such $I$. We call a multi-fan $\Delta=(\Sigma,C,w^\pm)$ of dimension $n$ \emph{pre-complete} if $\Sigma^{(n)}\neq\emptyset$ and the integer $d_v$ is independent of the choice of generic vectors $v$. We call this integer the \emph{degree} of $\Delta$.

For each $K\in\Sigma$, by setting $\Sigma_K:=\{J\in\Sigma\mid K\leq J\}$ and defining $N^{C(K)}$ to be the quotient lattice of $N$ by the sublattice generated by $C(K)\cap N$, we can naturally define $$C_K\colon \Sigma_K\to\Cone(N^{C(K)})$$ and $$w^\pm_K\colon \Sigma_K^{(n-|K|)}\subset\Sigma^{(n)}\to\Z_{\geq0}.$$ The triple $\Delta_K:=(\Sigma_K,C_K,w_K^\pm)$ is a multi-fan in $N^{C(K)}$, and it is called the \emph{projected multi-fan} with respect to $K\in\Sigma$.
\begin{definition}
A pre-complete multi-fan $\Delta=(\Sigma,C,w^\pm)$ is said to be \emph{complete} if the projected multi-fan $\Delta_K$ is pre-complete for any $K\in\Sigma$.
\end{definition}

For an ordinary fan $\Delta$ of dimension $n$, we define $w^+$ to be $1$ on $n$-dimensional cones and $w^-$ to be zero.  Then the ordinary fan $\Delta$ can be regarded as a multi-fan and it is complete in the ordinary sense (i.e. the union of cones in $\Delta$ is the entire space $N_\R$) if and only if it is complete in the sense of multi-fan.
We will see examples of non-singular complete multi-fans in later sections.

We say that $M$ is a \emph{torus manifold} if it is an orientable, compact, connected, smooth manifold of dimension $2n$ with an effective action of an $n$-dimensional torus $T$ with nonempty fixed point set $M^T$.  Since $\dim T=\frac{1}{2}\dim M$,  $M^T$ is isolated. A closed, connected, codimension-two submanifold of $M$ is said to be \emph{characteristic} if it is a connected component of the set fixed pointwise under some circle subgroup of $T$ and contains at least one $T$-fixed point.
Since $M$ is compact, there are only finitely many characteristic submanifolds. We denote them by $M_i$, $i=1,\ldots,d$.
We set $$\Sigma(M):=\{I\subset\{1,\ldots,d\}\mid(\bigcap_{i\in I}M_i)^T\neq\emptyset\}$$
and add an empty set to $\Sigma(M)$ as a member.  The set $\Sigma(M)$ is an abstract simplicial complex of dimension $n-1$.

An orientation on $M$ together with an orientation on each characteristic submanifold of $M$ is called an \emph{omniorientation} on $M$.  Suppose that $M$ is omnioriented.  Then the normal bundle $\nu_i$ to $M_i$, which is a real plane bundle, has an orientation induced from the omniorientation and we may think of $\nu_i$ as a complex line bundle.  Then we obtain a unique element $v_i$ of $\Hom(S^1,T)= H_2(BT;\Z)$ characterized by these two conditions:
\begin{itemize}
\item $v_i(S^1)$ fixes $M_i$ pointwise;
\item $v_i(g)_*(\xi)=g\xi$ for $g\in S^1$ and $\xi\in \nu_i$,
\end{itemize}
where $v_i(g)_*$ denotes the differential of the diffeomorphism $v_i(g)\colon M\to M$ and $g\xi$ denotes the complex multiplication of $\xi$ by $g\in S^1\subset \C$.
Note that if the orientation on $M_i$ is reversed, then $v_i$ turns into $-v_i$ and $v_j$ $(j\not=i$) remains fixed, moreover $(w^+(I), w^-(I))$ turns into $(w^-(I),w^+(I))$ if $I\ni i$ and remains fixed otherwise.
If the orientation on $M$ is reversed, then $v_i$ turns into $-v_i$ for all $i=1,\dots,d$ and $(w^+(I), w^-(I))$ turns into $(w^-(I),w^+(I))$ for all $I\in \Sigma(M)^{(n)}$.

Adopting $H_2(BT;\Z)$ as the lattice $N$, we define a map  $$C(M)\colon \Sigma(M)\to \Cone(N)$$ by sending $I\in\Sigma(M)$ to the cone in $H_2(BT;\Z)\otimes\R=H_2(BT;\R)$ spanned by $v_i$'s $(i\in I)$ (and the empty set to $\{0\}$).  When $I\in \Sigma(M)^{(n)}$, the intersection $\bigcap_{i\in I}M_i$ is isolated and fixed by the $T$-action on $M$.  At each fixed point $p\in \bigcap_{i\in I}M_i$, the tangent space of $M$ at $p$ has two orientations: one is endowed by the orientation of $M$ and the other comes from the intersection of the oriented submanifolds $M_{i}$'s $(i\in I)$.  Denoting the ratio of the above two orientations by $\epsilon_p$, we define  $w(M)^+(I)$ to be the number of points $p\in \bigcap_{i\in I}M_i$ with $\epsilon_p=+1$ and similarly for $w(M)^-(I)$.

\begin{theorem}[\cite{ha-ma03}] \label{theo:2.0}
The multi-fan $\Delta(M):=(\Sigma(M),C(M),w(M)^\pm)$ of an omnioriented torus manifold $M$ is non-singular and complete.
\end{theorem}

The equivariant connected sum can be done for omnioriented torus manifolds, which we shall explain.  If $M$ is an omnioriented torus manifold, then the tangential representation $\tau_pM$ at $p\in M^T$ can be regarded as a complex $T$-representation in such a way that the orientation induced from the complex structure on each irreducible factor of $\tau_pM$, which is of real dimension two, is compatible with the orientations on $M$ and $M_i$.  Let $M$ and $M'$ be omnioriented torus manifolds of dimension $2n$.  Suppose that there are fixed points $p\in M$ and $p'\in M'$ where the tangential representations are isomorphic as complex representations but the ratios $\epsilon_{p}$ and $\epsilon_{p'}$ mentioned above have opposite signs.  Then one can perform the equivariant connected sum of $M$ and $M'$ at $p$ and $p'$, denoted $M\sharp_{p,p'}M'$, in a natural way.  The set of cones in the multi-fan $\Delta(M\sharp_{p,p'}M')$ is the disjoint union of cones in $\Delta(M)$ and $\Delta(M')$ except the cones $C$ and $C'$ corresponding to $p$ and $p'$ and the cones $C$ and $C'$ are identified in $\Delta(M\sharp_{p,p'}M')$.  The weights on cones in $\Delta(M\sharp_{p,p'}M')$ are inherited from $\Delta(M)$ and $\Delta(M')$ except the identified cone and the weight on the identified cone is the sum of the weights on $C$ and $C'$.  When the sum is zero (meaning that both $w^+$ and $w^-$ are zero on the cone), we think that the identified cone vanishes.

One can blow up an omnioriented torus manifold $M$ at any fixed point $p$.  There is an $n$-dimensional cone $C$ in $\Delta(M)$ which corresponds to the fixed point $p$.  Blowing up at $p$ corresponds to adding an edge whose primitive vector is the sum of the primitive vectors on the edges in $C$ and making a stellar subdivision of $C$. The weights on the subdivided $n$-dimensional cones are all $(1,0)$ or all $(0,1)$ according as the ratio $\epsilon_p$ is $+1$ or $-1$.  When the sum $w^++w^-$ of the weight $(w^+,w^-)$ is more than one, the cone $C$ remains but $(1,0)$ or $(0,1)$ is subtracted from the weight $(w^+,w^-)$ on $C$ according as $\epsilon_p$ is $+1$ or $-1$.

Compact smooth toric varieties (with restricted compact torus actions) and quasitoric manifolds introduced by Davis-Januszkiewicz \cite{da-ja91} are torus manifolds.  These two families are contained in the family of topological toric manifolds introduced by Ishida-Fukukawa-Masuda \cite{is-fu-ma13} and topological toric manifolds (with restricted compact torus actions) are also torus manifolds.  The cohomology ring of a topological toric manifold is generated by degree two elements.  Another typical example of a torus manifold is the unit sphere $S^{2n}$ of $\C^n\oplus\R$, where the $T$-action on $\C^n$ is standard and that on $\R$ is trivial.

The $T$-action on a torus manifold $M$ is said to be \emph{locally standard} if it is locally modeled by the standard $T$-action on $\C^n$, to be precise, if any point of $M$ has an open $T$-invariant neighborhood equivariantly diffeomorphic to an $T$-invariant open subset of a faithful $T$-representation space of real dimension $2n$.  It is known that the $T$-action on $M$ is locally standard if $H^{odd}(M)=0$ (\cite{ma-pa06}).   If the $T$-action on $M$ is locally standard, then the orbit space $M/T$ is a manifold with corners.  The $T$-actions on the torus manifolds mentioned in the paragraph above are all locally standard and their orbit spaces are contractible; even every face of the orbit spaces is contractible.  However, the $T$-action is not necessarily locally standard for a general torus manifold $M$ and even if it is locally standard, the orbit space $M/T$ is not necessarily acyclic.  Here is a relationship between the topology of $M$ and $M/T$.

\begin{theorem}[\cite{ma-pa06}] \label{theo:2.1}
Let $M$ be a torus manifold with locally standard $T$-action. Then the following hold.
\begin{enumerate}
\item $H^{odd}(M)=0$ if and only if $M/T$ is \emph{face-acyclic}, i.e., every face of $M/T$ (even $M/T$ itself) is acyclic.
\item $H^*(M)$ is generated by degree two elements as a ring if and only if  $M/T$ is a \emph{homology polytope}, i.e., $M/T$ is face-acyclic and any intersection of faces of $M/T$ is connected unless empty.
\end{enumerate}
\end{theorem}

As for fundamental groups, we have the following.
\begin{lemma} \label{lemm:2.1}
Let $M$ be a torus manifold with locally standard $T$-action. Then the quotient map $q\colon M\to M/T$ induces an isomorphism $q_*\colon \pi_1(M)\to \pi_1(M/T)$ on their fundamental groups.
\end{lemma}
\begin{proof}
This is proved in \cite[Lemma 2.7]{wiem11}.  Since the proof is easy and used later, we shall give it.

Since the $T$-action on $M$ is locally standard, $M/T$ is a manifold with boundary $\partial(M/T)$.
Let $M^\circ$ be the union of all principal orbits in $M$. Then $M^\circ$ is a principal $T$-bundle over $M^\circ/T=M/T-\partial(M/T)$.
We consider the following commutative diagram
\begin{equation} \label{eq:cd}
\begin{CD}
\pi_1(T) @>\varphi_*>>\pi_1(M^\circ)@> q^\circ_*>> \pi_1(M^\circ/T)\\
@.  @V \iota_* VV                 @VV \bar{\iota}_*V\\
@.  \pi_1(M)@>q_*>>\pi_1(M/T)
\end{CD}
\end{equation}
where $\varphi\colon T\to Tx\subset M^\circ$ $(x\in M^\circ)$, $q^\circ$ is the restriction of $q$ to $M^\circ$ and $\iota$, $\bar{\iota}$ are both the inclusion maps.  We have $M^\circ = M \backslash \bigcup_{\{e\}\not=G\subset T}M^G$, where the union is taken over all non-trivial subtori $G$ of $T$. Because each $M^G$ has at least codimension-2 in $M$, it follows that
\begin{equation} \label{eq:cd1}
\text{$\iota_*$ in \eqref{eq:cd} is an epimorphism.}
\end{equation}
Clearly $\bar{\iota}\colon M^\circ/T\to M/T$ is a homotopy equivalence so that
\begin{equation} \label{eq:cd2}
\text{$\bar{\iota}_*$ in \eqref{eq:cd} is an isomorphism.}
\end{equation}
Since $q^\circ\colon M^\circ\to M^\circ/T$ is a principal $T$-bundle, $q^\circ_*$ in \eqref{eq:cd} is an epimorphism and $\ker q^\circ_*=\im \varphi_*$.  The free orbit $Tx$ shrinks to a fixed point in $M$ (one can see this easily if one takes the point $x$ to be close to a fixed point), so the composition $\iota_*\circ\varphi_*$ is trivial.
This observation together with \eqref{eq:cd1} and \eqref{eq:cd2} shows that $q_*$ in \eqref{eq:cd} is an isomorphism.
\end{proof}

\section{Toric origami manifolds}\label{sec:3}

In this section, we recall the definitions and properties of toric origami manifolds and origami templates.  Details can be found in \cite{ca-gu-pi11}.

A \emph{folded symplectic form} on a $2n$-dimensional manifold $M$ is a closed $2$-form $\omega$ whose top power $\omega^n$ vanishes transversally on a subset $Z$ and whose restriction to points in $Z$ has maximal rank. Then $Z$ is a codimension-one submanifold of $M$ and is called the \emph{fold}.  When the fold $Z$ is empty, $\omega$ is a genuine symplectic form.  The pair $(M,\omega)$ is called a \emph{folded symplectic manifold}.  An analog of Darboux's theorem for folded symplectic forms says that near any point $p$ of $Z$ there is a coordinate chart centered at $p$ where the form $\omega$ is
\[
x_1dx_1\wedge dy_1+dx_2\wedge dy_2+\dots+dx_n\wedge dy_n.
\]

Since the restriction of $\omega$ to $Z$ is assumed to have maximal rank, it has a one-dimensional kernel at each point of $Z$ and determines a line field on $Z$ called the \emph{null foliation}.  If the null foliation is the vertical bundle of some principal $S^1$-fibration $Z\to B$ over a compact base $B$, then the folded symplectic form $\omega$ is called an \emph{origami form} and the pair $(M,\omega)$ is called an \emph{origami manifold}.

The action of a Lie group $G$ on an origami manifold $(M,\omega)$ is \emph{Hamiltonian} if it admits a \emph{moment map} $\mu\colon M\to \mathfrak g^\ast$ satisfying the conditions:
\begin{itemize}
\item $\mu$ is equivariant with respect to the given action of $G$ on $M$ and the coadjoint action of $G$ on the vector space $\mathfrak g^\ast$ dual to the Lie algebra $\mathfrak g$ of $G$;
\item $\mu$ collects Hamiltonian functions, that is, $d\langle \mu,X\rangle=\iota_{X^\sharp}\omega$ for any $X\in \mathfrak g$, where $X^\sharp$ is the vector field on $M$ generated by $X$.
\end{itemize}

\begin{definition}
A \emph{toric origami manifold} $(M,\omega,T,\mu)$, abbreviated as $M$,  is a compact connected origami manifold $(M,\omega)$ equipped with an effective Hamiltonian action of a torus $T$ with $\dim T=\frac{1}{2}\dim M$ and with a choice of a corresponding moment map $\mu$.
\end{definition}

When the fold $Z$ is empty, a toric origami manifold is a symplectic toric manifold.  A famous theorem by Delzant \cite{delz88} says that symplectic toric manifolds are classified by their moment images called \emph{Delzant polytopes}.

The moment data of a toric origami manifold can be encoded into an \emph{origami template} $(\cP,\cF)$, where $\cP$ is a (nonempty) finite collection of $n$-dimensional Delzant polytopes in $\R^n$ and $\cF$ is a collection of facets and pairs of facets of polytopes in $\cP$ satisfying the following properties:
\begin{enumerate}
\item[(O1)] for each pair $\{F,F'\}\in\cF$, the corresponding polytopes $P$ and $P'$ in $\cP$ agree near those facets, that is, there is an open set $\mathcal{U}$ of $\R^n$ such that $\mathcal{U}\cap P=\mathcal{U}\cap P'$;\label{item:pair of folding facets}
\item[(O2)] if a facet $F$ occurs in $\cF$, either by itself or as a member of a pair, then neither $F$ nor any of its neighboring facets occur elsewhere in $\cF$;\label{item:property of folding facets}
\item[(O3)] the topological space, denoted $|(\cP,\cF)|$, constructed from the disjoint union $\sqcup P_j$, $P_j\in\cP$, by identifying facet pairs in $\cF$ is connected.
\end{enumerate}

The following is a generalization of the theorem by Delzant to toric origami manifolds.

\begin{theorem}[\cite{ca-gu-pi11}] \label{theo:3.1}
Assigning the moment data of a toric origami manifold induces a one-to-one correspondence
\begin{equation*}
\left\{\text{toric origami manifolds} \right\}\leftrightsquigarrow\{\text{origami templates}\}
\end{equation*}
up to equivariant origami symplectomorphism on the left-hand side, and affine equivalence on the right-hand side.
\end{theorem}

\begin{remark}
A toric origami manifold is not necessarily orientable, e.g. $\R P^{2n}$ admits a toric origami form (\cite{ca-gu-pi11}).  If a toric origami manifold is orientable, then there is no single facet in $\cF$ of the associated origami template $(\cP,\cF)$.  An origami template is said to be \emph{oriented} if the polytopes in $\cP$ come with an orientation and $\cF$ consists solely of pairs of facets which belong to polytopes with opposite orientations. The correspondence in Theorem~\ref{theo:3.1} also induces a one-to-one correspondence between oriented ones.
\end{remark}

\begin{example} \label{exam:3.1}
The unit sphere $S^{2n}\subset \C^n\oplus\R$ admits a toric origami form with the fold $S^{2n-1}\subset\C^n\oplus\{0\}$, where the origami form and the $T^n$-action on $S^{2n}$ is the restriction of the standard form $\sum_{i=1}^ndx_i\wedge dy_i$ and the standard $T^n$-action on $\C^n\oplus \R$ to $S^{2n}$.  Take $n=2$ and let $\cP$ be the set of two right-angled isosceles triangles with opposite orientations, and let $\cF$ have only one element, the pair of hypotenuse. Then $(\cP,\cF)$ is the oriented origami template corresponding to $S^4$, see \cite{ca-gu-pi11} for details.
\end{example}

Let $(\cP,\cF)$ be an origami template and $M$ be the associated toric origami manifold.
The topological space $|(\cP,\cF)|$ is a manifold with corners with the face structure induced from the face structures on polytopes in $\cP$, and $|(\cP,\cF)|$ is homeomorphic to $M/T$ as manifolds with corners.
The \emph{graph underlying the origami template}, denoted $\Gamma(\cP,\cF)$, has vertices corresponding to polytopes in $\cP$ and edges corresponding to the pairs in $\cF$. The graph $\Gamma(\cP,\cF)$ is connected by (O3) and has the same homotopy type as $|(\cP,\cF)|$, so that the orbit space $M/T$ is contractible or homotopy equivalent to a bouquet of $S^1$. The origami template $(\cP,\cF)$ is called \emph{acyclic} if the graph $\Gamma(\cP,\cF)$ is acyclic, i.e., a tree, and in this case $|(\cP,\cF)|$ has a vertex, equivalently $M$ has a fixed point.

An origami template $(\cP,\cF)$ is called \emph{co\"orientable} if $\cF$ consists only of pairs of facets of polytopes in $\cP$.  We say that a toric origami manifold is \emph{co\"orientable} if the associated origami template is co\"orientable.  An orientable origami template is co\"orientable (in other words, an orientable toric origami manifold is co\"orientable) but the converse does not hold (see \cite[Fig. 9]{ca-gu-pi11}).

\begin{theorem}[\cite{ho-pi12}] \label{theo:3.2}
Let $M$ be a co\"orientable toric origami manifold.  Then the $T$-action on $M$ is locally standard and if the origami template associated to $M$ is acyclic, then $H^{odd}(M)=0$.  Moreover, $M/T$ is face-acyclic if and only if the associated origami template is acyclic.
\end{theorem}

\begin{remark} The main part in Theorem~\ref{theo:3.2} also follows from Theorem~\ref{theo:2.1} as noted in \cite{ho-pi12}.
\end{remark}

Since $T$ is connected, the quotient map $q\colon M\to M/T$ induces an epimorphism $q_*\colon \pi_1(M)\to \pi_1(M/T)$ (see \cite[Corollary 6.3 in p.91]{bred72}).
The epimorphism $q_*$ is often an isomorphism (see Lemma~\ref{lemm:2.1}) but not always an isomorphism even if $M$ has a fixed point.  For instance, the orbit space of $\R P^{2n}$ by a standard $T$-action is contractible while $\pi_1(\R P^{2n})$ is of order two.
\begin{proposition} \label{prop:3.1}
If a toric origami manifold $M$ has a fixed point and is co\"orientable, then the quotient map $q\colon M\to M/T$ induces an isomorphism $q_*\colon \pi_1(M)\to \pi_1(M/T)$ and hence $\pi_1(M)$ is a free group.  Moreover, a co\"orientable toric origami manifold is simply connected if and only if the associated origami template is acyclic.
\end{proposition}

\begin{proof}
Since $M$ is co\"orientable, the $T$-action on $M$ is locally standard by Theorem~\ref{theo:3.2}.  Therefore the former statement in the proposition follows from Lemma~\ref{lemm:2.1}.  Precisely speaking, a torus manifold is assumed to be orientable while a co\"orientable toric origami manifold is not necessarily orientable as remarked above.  However, the proof for Lemma~\ref{lemm:2.1} works without orientability of $M$.  Since $M/T$ is homotopy equivalent to a bouquet of $S^1$ and $\pi_1(M)$ is isomorphic to $\pi_1(M/T)$, $\pi_1(M)$ is a free group.

The latter statement in the proposition follows from the former because the toric origami manifold associated to an acyclic origami template has a fixed point as remarked before.
\end{proof}

In the next section, we will discuss $\pi_1(M)$ of an orientable toric origami manifold $M$ which may have no fixed point.

\section{Multi-fans associated to oriented origami templates}\label{sec:4}

When $M$ is a symplectic toric manifold, each characteristic submanifold of $M$ is a symplectic manifold with the restricted symplectic form.   Therefore $M$ has a canonical omniorientation induced from the symplectic form and the multi-fan of $M$ is nothing but the normal fan of the Delzant polytope associated to $M$.  Similarly, one can find the multi-fan of an oriented toric origami manifold $M$ (with some omniorientation) from the oriented origami template associated to $M$.  We shall explain this in this section.

We introduce an operation on multi-fans.  For a non-singular complete multi-fan $\Delta:=(\Sigma,C,w^\pm)$ and an edge $L$, that is one-dimensional cone in $\Delta$, we denote by $\nbd(L)$ the set of all cones containing $L$.
Let $\Delta':=(\Sigma',C',{w'}^\pm)$ be another multi-fans of the same dimension as $\Delta$. Suppose that edges $L$ and $L'$ in $\Delta$ and $\Delta'$ satisfy the following:
\begin{equation*}\label{assumption on nbd of edge}
\nbd(L)=\nbd(L'), \quad w^\pm(I)={w'}^\mp(I') \mbox{ for }C(I)=C'(I')\in\nbd(L)=\nbd(L').
\end{equation*}
Then we obtain a new multi-fan $\Delta\dia_{L,L'}\Delta'$ by removing $\nbd(L)$ and $\nbd(L')$ from $\Delta$ and $\Delta'$ and then identifying the boundary cones in a natural way. We call $\dia_{L,L'}$ the \emph{diamond operation} at $L$ and $L'$.

\begin{example} \label{ex:diam}
Consider the normal fan of the square
$$\left\{(x,y)\in\R^2\,\mid \, -1\leq x\leq 1, -1\leq y\leq 1\right\}.$$
Let $\Delta$ be the normal fan with weight functions $w^+=1\mbox{ and }w^-=0$ on the $2$-dimensional cones. If we give opposite weight functions ${w'}^\pm(=w^\mp)$ to the $2$-dimensional cones, then we get another non-singular complete multi-fan, denoted $\Delta'$.  Let $L$ and $L'$ be the edge generated by $(0,-1)$ in $\Delta$ and $\Delta'$. Then $\Delta\dia_{L,L'}\Delta'$ is a non-singular complete multi-fan where the 2-dimensional cones are two first quadrants with opposite weight functions and two second quadrants with opposite weight functions.
\end{example}

\begin{definition}
We define a multi-fan associated to an oriented Delzant polytope $P$ to be the normal fan of $P$ with $(w^+,w^-)=(1,0)$ if $P$ is positively oriented and $(w^+,w^-)=(0,1)$ otherwise.
For an oriented origami template $(\cP,\cF)$, we define $\Delta(\cP,\cF)$ to be the multi-fan obtained by performing diamond operations on the non-singular complete multi-fans associated to the oriented Delzant polytopes in $\cP$ along the edges corresponding to the facets in $\cF$.
\end{definition}

An oriented origami template $(\cP,\cF)$ is co\"oriented, so the $T$-action on the toric origami manifold $M$ associated to $(\cP,\cF)$ is locally standard by Theorem~\ref{theo:3.2}.  If any intersection of facets of $|(\cP,\cF)|$ is connected and contains a vertex, then $\Delta(\cP,\cF)$ agrees with the multi-fan $\Delta(M)$ of $M$ with an appropriate omniorientation.  In this case, any cone is contained in an $n$-dimensional cone, where $n$ is the dimension of $(\cP,\cF)$, and the weight $(w^+,w^-)$ on each $n$-dimensional cone is either $(1,0)$ or $(0,1)$. However, $\Delta(\cP,\cF)$ may be different from $\Delta(M)$ as shown in the following example.

\begin{example}\label{ex:S4}
As mentioned in Example~\ref{exam:3.1}, the oriented origami template $(\cP,\cF)$ corresponding to $S^4$ consists of two copies of a right-angled isosceles triangle whose fold facet is a hypotenuse. Then the multi-fan $\Delta(\cP,\cF)$ has two edges and two $2$-dimensional cones with weights $(1,0)$ and $(0,1)$ respectively. In Figure~\ref{fig:template and multi-fan of S4}, the two $2$-dimensional cones are respectively black and red hatched.  On the other hand, the multi-fan $\Delta(S^4)$ of an omnioriented $S^4$ with the standard $T^2$-action has two edges and one 2-dimensional cone with weight $(w^+,w^-)=(1,1)$. The reason why $\Delta(\cP,\cF)$ and $\Delta(S^4)$ are different is that the intersections of the two facets of $|(\cP,\cF)|$ is disconnected, actually it consists of two vertices.  Note that if we identify the two $2$-dimensional cones in $\Delta(\cP,\cF)$ and assign the sum of the weights on the cones to the identified cone, then we will obtain $\Delta(S^4)$.
        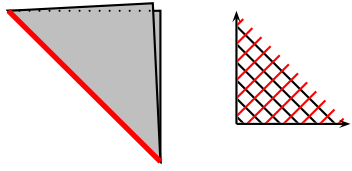
\begin{figure}[ht]
            \psset{unit=.5cm}
            \begin{center}
            \begin{pspicture}(0,0)(9,4)
                \pspolygon[fillstyle=solid,fillcolor=lightgray](0,4)(4,0)(4,4)
                \pspolygon[fillstyle=solid,fillcolor=lightgray](0,4)(4,0)(3.8,4.2)
                \psline[linestyle=dotted](0,4)(4,4)
                %\psline(0,4)(3.8,4.2)
                %\psline(3.8,4.2)(4,0)
                \psline[linewidth=2pt,linecolor=red](4,0)(0,4)
                \pspolygon[fillstyle=hlines,hatchangle=135,linecolor=white](6,1)(6,4)(9,1)
                \pspolygon[fillstyle=hlines,hatchangle=45,hatchcolor=red,linecolor=white](6,1)(6,4)(9,1)
                \psline{->}(6,1)(6,4)
                \psline{->}(6,1)(9,1)
            \end{pspicture}
            \end{center}
            \caption{Origami template and multi-fan of $S^4$}
            \label{fig:template and multi-fan of S4}
        \end{figure}
        \end{example}

The multi-fan $\Delta(\cP,\cF)$ in Example~\ref{ex:S4} is complete although it is different from $\Delta(S^4)$.  However, $\Delta(\cP,\cF)$ may not be complete as shown in the following example.
\begin{example} \label{exam:4.2}
Let $P$ be the square in Example~\ref{ex:diam} and let $F^+$ and $F^-$ be the sides of $P$ defined by $y=1$ and $y=-1$ respectively.  Let $\bar P, \bar F^\pm$ be copies of $P, F^\pm$ respectively.  Then $\cP=\{P, \bar P\}$ and $\cF^-=\{\{F^-,\bar F^-\}\}$ form an origami template and $\Delta(\cP,\cF^-)$ is the non-singular complete multi-fan in Example~\ref{ex:diam}.  However, if $\cF^\pm=\{\{F^-,\bar F^-\}, \{F^+, \bar F^+\}\}$, then $\Delta(\cP,\cF^\pm)$ has two edges but no $2$-dimensional cone.  Therefore $\Delta(\cP,\cF^\pm)$ is not complete because the existence of an $n$-dimensional cone is required in the definition of completeness for an $n$-dimensional multi-fan, see Section~\ref{sec:2}. Note that the toric origami manifold associated to $(\cP,\cF^\pm)$ has no fixed point, so it is not a torus manifold because the existence of a fixed point is required in the definition of a torus manifold, see Section~\ref{sec:2}.  In fact, the toric origami manifold associated to $(\cP,\cF^\pm)$ is equivariantly diffeomorphic to the product of $S^2$ with the standard $S^1$-action and the 2-dimensional torus with a free $S^1$-action.
\end{example}

Let $M$ be an orientable toric origami manifold.  We choose an orientation on $M$ and let $(\cP,\cF)$ be the oriented origami template associated to $M$ with the orientation.  The multi-fan $\Delta(\cP,\cF)$ is defined in $N_\R=N\otimes\R$ where $N=H_2(BT;\Z)$ is a lattice of rank $n=\dim T$. Let $N_\Delta$ be the sublattice of $N$ generated by primitive vectors sitting in one-dimensional cones in $\Delta(\cP,\cF)$. Note that $N_\Delta$ is independent of the choice of the orientation on $M$.  If $\Delta(\cP,\cF)$ has an $n$-dimensional cone (equivalently if $M$ has a fixed point), then $N_\Delta$ agrees with $N$.  Otherwise $N_\Delta$ may not agree with $N$ but the rank of $N_\Delta$ is at least $n-1$ because $\Delta(\cP,\cF)$ contains an $(n-1)$-dimensional cone.  More precisely, since any $(n-1)$-dimensional cone in $\Delta(\cP,\cF)$ is a facet of a non-singular $n$-dimensional cone (because they are associated to Delzant polytopes in $\cP$), the quotient group $N/N_\Delta$ is a finite or infinite cyclic group.

\begin{proposition} \label{prop:4.1}
Let $M$ be an orientable toric origami manifold and let $N_\Delta$ be as above.  Let $q_*\colon \pi_1(M)\to \pi_1(M/T)$ be the homomorphism induced from the quotient map $q\colon M\to M/T$.
Then there is an epimorphism
\[
\psi\colon N/N_\Delta \times \pi_1(M/T) \to \pi_1(M)
\]
such that the composition $q_*\circ \psi \colon N/N_\Delta \times \pi_1(M/T)\to \pi_1(M/T)$ is the projection on the second factor, in particular,
$\ker\psi$ is contained in $N/N_\Delta$.
\end{proposition}

\begin{proof}
In the following, we use the notations in the proof of Lemma~\ref{lemm:2.1} freely.  As remarked in the paragraph after  Example~\ref{exam:3.1}, $M/T$ is homotopy equivalent to a bouquet of $S^1$ and so is $M^\circ/T$.  Therefore, $H^2(M^\circ/T;\Z)$ vanishes and the principal $T$-bundle $q^\circ\colon M^\circ \to M^\circ/T$ in \eqref{eq:cd} is trivial, so
$\pi_1(M^\circ)\cong \pi_1(T)\times\pi_1(M^\circ/T)$.  Since $\pi_1(M^\circ/T)\cong \pi_1(M/T)$ via $\bar\iota_*$, $\iota_*$ induces an epimorphism
\begin{equation} \label{eq:psi'}
\psi'\colon \pi_1(T)\times \pi_1(M/T)\to \pi(M)
\end{equation}
such that the composition $q_*\circ\psi'$ is the projection on the second factor.
Since $N=H_2(BT;\Z)=\Hom(S^1,T)$, $\pi_1(T)$ can be identified with $N$.  The circle subgroup $S_i$ which fixes $M_i$ pointwise corresponds to a primitive vector $v_i$ in $N=\Hom(S^1,T)$ such that $v_i(S^1)=S_i$.  The cone spanned by $v_i$ belongs to $\Delta(\cP,\cF)$ and any one-dimensional cone in $\Delta(\cP,\cF)$ is obtained in this way.  Since the orbit $S_ix$ $(x\in M^\circ)$ shrinks to a point in $M$, the subgroup $N_\Delta$ maps to the identity element in $\pi_1(M)$ through $\psi'$ in \eqref{eq:psi'}. Thus $\psi'$ induces the desired epimorphism $\psi$ in the proposition.
\end{proof}

\begin{example}
The toric origami manifold $M$ associated to the origami template $(\cP,\cF^\pm )$ in Example~\ref{exam:4.2} is equivariantly diffeomorphic to the product of $S^2$ with the standard $S^1$-action and the 2-dimensional torus with a free $S^1$-action.  In this case, $\pi_1(M)\cong\Z\oplus\Z$, $\pi_1(M/T)\cong\Z$ and $\ker q_*$ is an infinite cyclic group while $N_\Delta$ is a sublattice of rank one and $N/N_\Delta$ is also an infinite cyclic group.
\end{example}

\begin{corollary} \label{coro:4.1}
The fundamental group of a non-simply connected orientable toric origami manifold $M$ is isomorphic to the product of a finite or infinite cyclic group (possibly trivial) and a non-trivial free group.
\end{corollary}

\begin{proof}
The corollary almost follows from Proposition~\ref{prop:4.1}.  The only thing we have to check is that a non-trivial finite cyclic group does not occur as $\pi_1(M)$.  If it does, then $M/T$ must be simply connected because $q_*\colon \pi_1(M)\to \pi_1(M/T)$ is an epimorphism and $\pi_1(M/T)$ is a free group.  Therefore the origami template associated to $M$ is acyclic and hence $M$ is simply connected by Proposition~\ref{prop:3.1}, a contradiction.
\end{proof}

For an oriented origami template $(\cP,\cF)$, we say that a facet of a Delzant polytope in $\cP$ is \emph{non-folded} if it does not touch any facet in $\cF$ and denote its neighborhood in $P$ by $\nbd(F)$.  Let $(\cP',\cF')$ be another oriented origami template of the same dimension as $(\cP,\cF)$ and assume that there are non-folded facets $F$ in $P\in \cP$ and $F'$ in $P'\in\cP'$ such that
\begin{equation} \label{eq:4.1}
\nbd(F)=\nbd(F'), \quad\text{$P$ and $P'$ have opposite orientations}.
\end{equation}
Then we define a new oriented origami template $(\cP,\cF)\dia_{F,F'}(\cP',\cF')$ by
\[
(\cP,\cF)\dia_{F,F'}(\cP',\cF'):=(\cP\sqcup \cP', \cF\sqcup \cF'\sqcup \{F,F'\})
\]
Clearly from the definition we have
\[
\Delta((\cP,\cF)\dia_{F,F'}(\cP',\cF'))=\Delta(\cP,\cF)\dia_{L,L'}\Delta(\cP',\cF')
\]
where $L$ and $L'$ are the edges corresponding to $F$ and $F'$.

We shall interpret the operation above in terms of oriented toric origami manifolds.
Let $M$ and $M'$ be toric origami manifolds whose oriented origami templates, denoted $(\cP,\cF)$ and $(\cP',\cF')$ respectively, satisfy \eqref{eq:4.1}. We denote by $M\dia_{F,F'} M'$ the oriented toric origami manifold corresponding to $(\cP,\cF)\dia_{F,F'}(\cP',\cF')$. Topologically, $M\dia_{F,F'} M'$ can be obtained from $M$ and $M'$ as follows.  Let $\tilde F$ and $\tilde F'$ be the characteristic submanifold of $M$ and $M'$ corresponding to $F$ and $F'$ respectively.  Assumption \eqref{eq:4.1} means that $\tilde F$ and $\tilde F'$ have the same neighborhood with opposite orientations on the normal bundles of $\tilde F$ and $\tilde F'$.  Then one can perform the connected sum of $M$ and $M'$ equivariantly along $\tilde F$ and $\tilde F'$.  The resulting manifold is equivariantly diffeomorphic to $M\dia_{F,F'} M'$.

Let $M$ and $M'$ be omnioriented toric origami manifolds. Remember that if there are fixed points $p\in M$ and $p'\in M'$ where the tangential representations are isomorphic as complex representations but the ratios $\epsilon_{p}$ and $\epsilon_{p'}$ induced from the omniorientations on $M$ and $M'$ have opposite signs, then we can perform the equivariant connected sum $M\sharp_{p,p'}M'$, see Section~\ref{sec:2}.  This equivariant connected sum can be described in terms of a blow up and a diamond operation as follows.  In terms of the origami templates, we cut the vertices corresponding to the fixed points and glue the created facets.  In terms of the multi-fans, we make a stellar subdivision of the $n$-dimensional cones corresponding to the fixed points and perform the diamond operation at the created edges.

\section{$4$-dimensional case}\label{sec:4-dimensional case}

As is well-known, any symplectic toric manifold of dimension $4$ is diffeomorphic to $\C P^1\times \C P^1$ or $\C P^2\sharp q\overline{\C P^2}$ with some non-negative integer $q$, where $\overline{\C P^2}$ denotes $\C P^2$ with reversed orientation.
It is also well-known that a simply connected compact smooth 4-manifold with an effective smooth action of $T^2$ is diffeomorphic to $S^4$ or $p \C P^2\sharp q\overline{\C P^2}\sharp r(\C P^1\times \C P^1)$ with $p+q+r\ge 1$ (\cite{or-ra70}).
The purpose of this section is to prove the following, which tells us that the class of toric origami manifolds is much larger than that of symplectic toric manifolds in dimension $4$.

\begin{theorem}\label{theo:5.1}
Any simply connected compact smooth $4$-manifold $M$ with an effective smooth action of $T^2$ is equivariantly diffeomorphic to a toric origami manifold.
\end{theorem}

It is known (and not difficult to see) that the action of $T^2$ on $M$ in the theorem above is locally standard and the orbit space $M/T^2$ is homeomorphic to a $2$-disk (\cite{or-ra70}).  We can number the characteristic submanifolds $M_i$ $(i=1,\dots,d)$ of $M$ in such a way that the intersection of $M_i$ and $M_{i+1}$ is nonempty, where $M_{d+1}=M_1$.  In fact, the intersection consists of two fixed points when $d=2$ and one fixed point when $d\ge 3$.

We choose an omniorientation on $M$.  Then we obtain the multi-fan $\Delta(M)$ of $M$.  The primitive edge vector $v_i$ corresponding to $M_i$ is an element of $H_2(BT^2)$ and we identify $H_2(BT^2)$ with $\Z^2$. Each pair $(v_i,v_{i+1})$ for $i=1,\dots,d$, where $v_{d+1}=v_1$, is a basis of $\Z^2$, so $\det(v_i,v_{i+1})$ is $\pm 1$.  Such a sequence $v_1,\dots,v_d$ is called \emph{unimodular} in \cite{hi-ma12}.  For each $i=1,\dots,d$, we form a two-dimensional cone $\angle v_iv_{i+1}$ spanned by $v_i$ and $v_{i+1}$ and assign $\det(v_i,v_{i+1})$ to the cone as the weight.  This produces a non-singular complete multi-fan and one sees that it agrees with $\Delta(M)$ if necessary by reversing the orientation on $M$.  Therefore, the unimodular sequence is essentially equivalent to the multi-fan.

Remember that if we reverse the orientation on the characteristic submanifold $M_i$, then $v_i$ turns into $-v_i$ but the other vectors remain fixed.  Therefore, the unimodular sequence $v_1,\dots,v_d$ is associated to $M$ only up to sign if we do not specify the omniorientation on $M$.  The unimodular sequence (up to sign) is important because of the following.

\begin{theorem}[\cite{or-ra70}] \label{theo:5.2}
The unimodular sequence (up to sign) associated to $M$ in Theorem~\ref{theo:5.1} determines the equivariant diffeomorphism type of $M$.
\end{theorem}

Therefore, in order to prove Theorem~\ref{theo:5.1}, it suffices to prove the following.

\begin{proposition} \label{prop:5.1}
The multi-fan formed from any unimodular sequence in $\Z^2$ is associated to an origami template if we change signs of the vectors appropriately.
    \end{proposition}
    \begin{proof}
 We prove the proposition by induction on the length $d$ of a unimodular sequence $v_1,\dots,v_d$.  Let $\Delta$ be the multi-fan formed from the unimodular sequence.  We may assume that $v_1=(1,0)$ and $v_2=(0,1)$ through a modular transformation on $\Z^2$.

When $d=2$, $\Delta$ is associated to an origami template as observed in Section~\ref{sec:4} (see Figure~\ref{fig:template and multi-fan of S4}).

When $d=3$, we have $v_3=(\pm 1,\pm 1)$.  When $v_3=(-1,-1)$, $\Delta$ is the fan of $\C P^2$, that is, the normal fan of a right-angled isosceles triangle.  In the remaining three cases, $\Delta$ can be obtained from the fan of $\C P^2$ and the fan of a Hirzebruch surface (that is the normal fan of a trapezoid) through diamond operation, see Figure~\ref{fig:d=3}, where $+$ or $-$ on a cone means that the weight $(w^+,w^-)$ on the cone is $(1,0)$ or $(0,1)$ respectively.

\begin{figure}[h]
\psset{unit=.5cm}
\begin{center}
\begin{pspicture}(-2,-5.5)(17.5,10)
    \pspolygon[fillstyle=hlines,hatchangle=135,linecolor=white](0,-3)(2,-3)(0,-1)
    \pspolygon[fillstyle=hlines,hatchangle=90,linecolor=white](0,-3)(2,-3)(2,-5)
    \pspolygon[fillstyle=hlines,hatchangle=45,hatchcolor=green,linecolor=white](0,-3)(2,-5)(2,-3)(0,-1)
    \psline{->}(0,-3)(2,-3)
    \rput(2.3,-3.5){$v_1$}
    \psline{->}(0,-3)(0,-1)
    \rput(0.3,-0.7){$v_2$}
    \psline{->}(0,-3)(2,-5.2)
    \rput(2.3,-5){$v_3$}
    \rput(4.5,-3){$=$}
    \pspolygon[fillstyle=hlines,hatchangle=135,linecolor=white](7,-3)(9,-3)(7,-1)
    \pspolygon[fillstyle=hlines,hatchangle=45,linecolor=white](7,-3)(7,-1)(5,-3)
    \pspolygon[fillstyle=hlines,hatchangle=150,linecolor=white](7,-3)(5,-3)(9,-5)
    \pspolygon[fillstyle=hlines,hatchangle=90,linecolor=white](7,-3)(9,-5)(9,-3)
    \psline{->}(7,-3)(9,-3)
    \rput(9.3,-3.5){$v_1$}
    \psline{->}(7,-3)(7,-1)
    \rput(7.3,-0.7){$v_2$}
    \psline[linecolor=blue]{->}(7,-3)(5,-3)
    \rput(5,-3.5){$\blue{\tilde{v}}$}
    \psline{->}(7,-3)(9,-5)
    \rput(9.3,-5.2){$v_3$}
    \rput(8,-2){$\red{+}$}
    \rput(6,-2){$\red{+}$}
    \rput(6.5,-4){$\red{+}$}
    \rput(8,-3.5){$\red{+}$}
    \rput(10,-2.5){$\dia_{L,L'}$}
    \pspolygon[fillstyle=hlines,hatchangle=45,hatchcolor=green,linecolor=white](13,-3)(15,-5)(13,-1)
    \pspolygon[fillstyle=hlines,hatchangle=135,linecolor=white](13,-3)(13,-1)(11,-3)
    \pspolygon[fillstyle=hlines,hatchangle=150,linecolor=white](13,-3)(11,-3)(15,-5)
    \psline{->}(13,-3)(13,-1)
    \rput(13.3,-0.7){$v_2$}
    \psline[linecolor=blue]{->}(13,-3)(11,-3)
    \rput(11,-3.5){$\blue{\tilde{v}}$}
    \psline{->}(13,-3)(15,-5)
    \rput(15.3,-5.2){$v_3$}
    \rput(14,-3){$\red{-}$}
    \rput(12,-2){$\red{-}$}
    \rput(12.5,-4){$\red{-}$}
    %%%%%%%%%%%%%%%%%%%%%%%%%%%%%%%%%%%%%%%%%%%%%%%%%%%%%%%%%%%%%%%%%%%%%%%%%%%%%%%%%%%%%%%%%%%%%%%
    \pspolygon[fillstyle=hlines,hatchangle=135,linecolor=white](0,2)(2,2)(2,4)
    \pspolygon[fillstyle=hlines,hatchangle=180,linecolor=white](0,2)(2,4)(0,4)
    \pspolygon[fillstyle=hlines,hatchangle=45,hatchcolor=green,linecolor=white](0,2)(0,4)(2,4)(2,2)
    \psline{->}(0,2)(2,2)
    \rput(2,1.5){$v_1$}
    \psline{->}(0,2)(0,4)
    \rput(0.3,4.2){$v_2$}
    \psline{->}(0,2)(2,4)
    \rput(2.5,4.2){$v_3$}
    \rput(4.5,2){$=$}
    \pspolygon[fillstyle=hlines,hatchangle=45,hatchcolor=green,linecolor=white](7,2)(7,4)(9,2)
    \pspolygon[fillstyle=hlines,hatchangle=30,linecolor=white](7,2)(5,0)(9,2)
    \pspolygon[fillstyle=hlines,hatchangle=70,linecolor=white](7,2)(5,0)(7,4)
    \psline{->}(7,2)(9,2)
    \rput(9,1.5){$v_1$}
    \psline{->}(7,2)(7,4)
    \rput(7.5,3.8){$v_2$}
    \psline[linecolor=blue]{->}(7,2)(5,0)
    \rput(5.5,0){$\blue{\tilde{v}}$}
    \rput(8,3){$\red{+}$}
    \rput(7,1){$\red{+}$}
    \rput(6,2){$\red{+}$}
    \rput(10.5,2){$\dia_{L,L'}$}
    \pspolygon[fillstyle=hlines,hatchangle=135,linecolor=white](13,2)(15,2)(15,4)
    \pspolygon[fillstyle=hlines,hatchangle=30,linecolor=white](13,2)(11,0)(15,2)
    \pspolygon[fillstyle=hlines,hatchangle=180,linecolor=white](13,2)(15,4)(13,4)
    \pspolygon[fillstyle=hlines,hatchangle=70,linecolor=white](13,2)(11,0)(13,4)
    \psline{->}(13,2)(15,2)
    \rput(15,1.5){$v_1$}
    \psline{->}(13,2)(15,4)
    \rput(15.5,4.2){$v_3$}
    \psline[linecolor=blue]{->}(13,2)(11,0)
    \rput(11.5,0){$\blue{\tilde{v}}$}
    \psline{->}(13,2)(13,4)
    \rput(13.5,3.8){$v_2$}
    \rput(15,3){$\red{-}$}
    \rput(14,3.5){$\red{-}$}
    \rput(12,2){$\red{-}$}
    \rput(13,1){$\red{-}$}
    %%%%%%%%%%%%%%%%%%%%%%%%%%%%%%%%%%%%%%%%%%%%%%%%%%%%%%%%%%%%%%%%%%%%%%%%%%%%%%%%%%%%%%%%%%%%%%%%%%%%%%%%
    \pspolygon[fillstyle=hlines,hatchangle=65,hatchcolor=green,linecolor=white](0,7)(2,7)(0,9)(-2,9)
    \pspolygon[fillstyle=hlines,hatchangle=30,linecolor=white](0,7)(0,9)(-2,9)
    \pspolygon[fillstyle=hlines,hatchangle=135,linecolor=white](0,7)(2,7)(0,9)
    \psline{->}(0,7)(2,7)
    \rput(2,6.5){$v_1$}
    \psline{->}(0,7)(0,9)
    \rput(0.3,9.2){$v_2$}
    \psline{->}(0,7)(-2,9)
    \rput(-2.5,9.2){$v_3$}
    \rput(4.3,7.5){$=$}
    \pspolygon[fillstyle=hlines,hatchangle=135,linecolor=white](7,7)(9,7)(7,9)
    \pspolygon[fillstyle=hlines,hatchangle=30,linecolor=white](7,7)(7,9)(5,9)
    \pspolygon[fillstyle=hlines,hatchangle=115,linecolor=white](7,7)(5,9)(7,5)
    \pspolygon[fillstyle=hlines,hatchangle=45,linecolor=white](7,7)(7,5)(9,7)
    \psline{->}(7,7)(9,7)
    \rput(9,6.5){$v_1$}
    \psline{->}(7,7)(7,9)
    \rput(7.3,9.2){$v_2$}
    \psline{->}(7,7)(5,9)
    \rput(4.5,8.7){$v_3$}
    \psline[linecolor=blue]{->}(7,7)(7,5)
    \rput(6.5,5.5){$\blue{\tilde{v}}$}
    \rput(8,8){$\red{+}$}
    \rput(8,6){$\red{+}$}
    \rput(6.5,8.5){$\red{+}$}
    \rput(6,7){$\red{+}$}
    \rput(10,7.5){$\dia_{L,L'}$}
    \pspolygon[fillstyle=hlines,hatchangle=45,linecolor=white](13,7)(13,5)(15,7)
    \pspolygon[fillstyle=hlines,hatchangle=115,linecolor=white](13,7)(13,5)(11,9)
    \pspolygon[fillstyle=hlines,hatchangle=65,hatchcolor=green,linecolor=white](13,7)(15,7)(11,9)
    \psline{->}(13,7)(15,7)
    \rput(15.5,7){$v_1$}
    \psline{->}(13,7)(11,9)
    \rput(11.3,9.2){$v_3$}
    \psline[linecolor=blue]{->}(13,7)(13,5)
    \rput(12.5,5.5){$\blue{\tilde{v}}$}
    \rput(13,8){$\red{-}$}
    \rput(14,6){$\red{-}$}
    \rput(12,7){$\red{-}$}
\end{pspicture}
\end{center}
\caption{}
\label{fig:d=3}
\end{figure}

Now we assume $d\ge 4$ and the proposition holds for unimodular sequences of length at most $d-1$.  As is shown in~\cite[Lemma 1.3]{hi-ma12}, there is $v_j$ $(1\le j\le d)$ which satisfies
        \[
        \epsilon_{j-1}v_{j-1}+\epsilon_j v_{j+1}+a_jv_j=0 \quad (\text{with $a_j=\pm 1$ or $0$}).
        \]
        In fact, $v_j$ is the vector whose Euclidean norm is maximal among $v_1,\ldots,v_d$. We distinguish two cases.

        \emph{The case where $a_j=\pm 1$}. We remove $v_{j}$ from the sequence $v_1,\dots,v_d$. The resulting sequence
\begin{equation} \label{eq:5.1}
v_1,\dots,v_{j-1}, v_{j+1},\dots, v_d
\end{equation}
is still unimodular, so we may assume that (if necessary by changing signs of some vectors in the original unimodular sequence) the multi-fan $\Delta_1$ formed from the sequence \eqref{eq:5.1} is associated to an origami template by induction assumption.  The multi-fan $\Delta_2$ obtained by blowing up $\Delta_1$ at the cone $\angle v_{j-1}v_{j+1}$ is also associated to an origami template because so is $\Delta_1$.

We may assume $j=2$ if necessary by changing the numbering of the vectors through a cyclic permutation.  If $v_2=v_1+v_3$ (that is, $v_3=(-1,1)$), then $\Delta=\Delta_2$.  Unless $v_2=v_1+v_3$, we have $v_3=(1,1), (1,-1)$ or $(-1,-1)$.  In each case we consider the multi-fan $\Delta_2'$ formed from a unimodular sequence $v_3,v_2,v_1, v_1+v_3$, namely $\Delta_2'$ is the multi-fan obtained by blowing up the multi-fan $\Delta_1'$ formed from a unimodular sequence $v_3,v_2,v_1$ at the cone $\angle v_1v_3$. Since $\Delta_1'$ is  associated to an origami template as observed above, so is
$\Delta_2'$.  Clearly $\Delta=\Delta_2\dia \Delta_2'$, where the diamond operation is performed at the edge spanned by $v_1+v_3$, so $\Delta$ is associated to an origami template because so are both $\Delta_2$ and $\Delta_2'$.

\emph{The case where $a_j=0$}.  In this case, we remove $v_{j-1}$ and $v_{j}$ from the sequence $v_1,\dots,v_d$. The resulting sequence
\begin{equation} \label{eq:5.2}
v_1,\dots,v_{j-2}, v_{j+1},\dots, v_d
\end{equation}
is still unimodular, so we may assume that (if necessary by changing signs of some vectors in the original unimodular sequence) the multi-fan $\Delta_1$ formed from the sequence \eqref{eq:5.2} is associated to an origami template by induction assumption.  The multi-fan $\Delta_2$ obtained by blowing up $\Delta_1$ at the cone $\angle v_{j-2}v_{j+1}$ is also associated to an origami template because so is $\Delta_1$.

We may assume $j=3$ if necessary by changing the numbering of the vectors through a cyclic permutation.  Then $v_2=\pm v_4$, but we may assume that $v_2=-v_4$ if necessary by changing the sign of $v_2$.  We consider the multi-fan $\Delta_2'$ formed from a unimodular sequence $v_4,v_3,v_2, v_1, v_1+v_4$.  Note that $v_3=(\pm 1,b)$ for some $b\in\Z$.

\medskip
\noindent
{\bf Claim.} $\Delta_2'$ is associated to an origami template.

\medskip
\noindent
\emph{Proof of the claim}. Since $\Delta'_2$ is obtained by blowing up the multi-fan $\Delta_1'$ formed from a unimodular sequence $v_4,v_3,v_2, v_1$ at the cone $\angle v_1v_4$, it suffices to check that $\Delta_1'$ is associated to an origami template.
When $v_3=(-1,b)$, $\Delta_1'$ is the fan corresponding to a Hirzebruch surface, so it is the normal fan of a trapezoid.  When $v_3=(1,b)$, we take $\tilde{v}=(-1,0)$. Then $v_1,v_2,\tilde{v},v_4$ (respectively, $v_3,v_4,\tilde{v},v_2$) is a unimodular sequence rotating around the origin once counterclockwise (respectively, clockwise). Let $\Delta'$ be the fan formed from the unimodular sequence $v_1,v_2,\tilde{v},v_4$ and $\Delta''$ the fan formed from the other unimodular sequence $v_3,v_4,\tilde{v},v_2$. Then $\Delta'_1=\Delta'\dia_{L',L''}\Delta''$, where $L'$ (respectively, $L''$) is the edge spanned by $\tilde{v}$ in $\Delta'$ (respectively, $\Delta''$), see Figure~\ref{fig:3}.  This proves the claim.

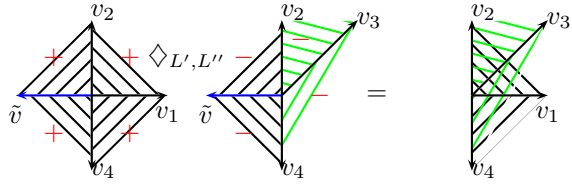
\begin{figure}[ht]
\psset{unit=.5cm}
\begin{center}
\begin{pspicture}(0,0)(14.5,4.5)
    \pspolygon[fillstyle=hlines,hatchangle=135,linecolor=white](2,2)(4,2)(2,4)
    \pspolygon[fillstyle=hlines,hatchangle=45,linecolor=white](2,2)(2,4)(0,2)
    \pspolygon[fillstyle=hlines,hatchangle=135,linecolor=white](2,2)(0,2)(2,0)
    \pspolygon[fillstyle=hlines,hatchangle=45,linecolor=white](2,2)(2,0)(4,2)
    \psline{->}(2,2)(4,2)
    \rput(4,1.5){$v_1$}
    \psline{->}(2,2)(2,4)
    \rput(2.3,4.2){$v_2$}
    \psline[linecolor=blue]{->}(2,2)(0,2)
    \rput(0,1.5){$\tilde{v}$}
    \psline{->}(2,2)(2,0)
    \rput(2.3,0){$v_4$}
    \rput(1,3){$\red{+}$}
    \rput(3,3){$\red{+}$}
    \rput(3,1){$\red{+}$}
    \rput(1,1){$\red{+}$}
    \rput(4.5,3){$\dia_{L',L''}$}
    \pspolygon[fillstyle=hlines,hatchangle=165,hatchcolor=green,linecolor=white](7,2)(9,4)(7,4)
    \pspolygon[fillstyle=hlines,hatchangle=45,linecolor=white](7,2)(7,4)(5,2)
    \pspolygon[fillstyle=hlines,hatchangle=135,linecolor=white](7,2)(5,2)(7,0)
    \pspolygon[fillstyle=hlines,hatchangle=60,hatchcolor=green,linecolor=white](7,2)(7,0)(9,4)
    \psline{->}(7,2)(9,4)
    \rput(9.3,4){$v_3$}
    \psline{->}(7,2)(7,4)
    \rput(7.3,4.2){$v_2$}
    \psline[linecolor=blue]{->}(7,2)(5,2)
    \rput(5,1.5){$\tilde{v}$}
    \psline{->}(7,2)(7,0)
    \rput(7.3,0){$v_4$}
    \rput(6,3){$\red{-}$}
    \rput(6,1){$\red{-}$}
    \rput(8,2){$\red{-}$}
    \rput(7.5,3.5){$\red{-}$}
    \rput(9.5,2){$=$}
    \pspolygon[fillstyle=hlines,hatchangle=135,linecolor=white](12,2)(14,2)(12,4)
    \pspolygon[fillstyle=hlines,hatchangle=45,linecolor=white](12,2)(12,0)(14,2)
    \pspolygon[fillstyle=hlines,hatchangle=165,hatchcolor=green,linecolor=white](12,2)(14,4)(12,4)
    \pspolygon[fillstyle=hlines,hatchangle=60,hatchcolor=green,linecolor=white](12,2)(12,0)(14,4)
    \psline{->}(12,2)(14,2)
    \rput(14,1.5){$v_1$}
    \psline{->}(12,2)(14,4)
    \rput(14.3,4){$v_3$}
    \psline{->}(12,2)(12,4)
    \rput(12.3,4.2){$v_2$}
    \psline{->}(12,2)(12,0)
    \rput(12.3,0){$v_4$}
\end{pspicture}
\end{center}
\caption{$\Delta'\dia_{L',L''}\Delta''$}
\label{fig:3}
\end{figure}

\medskip
Clearly $\Delta=\Delta_2\dia \Delta_2'$, where the diamond operation is performed at the edge spanned by $v_1+v_4$, so $\Delta$ is associated to an origami template because so are both $\Delta_2$ and $\Delta_2'$.  This completes the proof of the proposition.
\end{proof}

\begin{remark}
(1) The reader might expect that the multi-fan $\Delta$ formed from a unimodular sequence $v_1,\dots,v_d$ is always associated to an origami template without changing signs of the vectors.  This is the case when $d=2,3$ as observed in the proof of  Proposition~\ref{prop:5.1} but not the case when $d=4$.  For instance, one can see that the multi-fan $\Delta$ formed from a unimodular sequence $v_1=(1,0)$, $v_2=(0,1)$, $v_3=(1,b)$ $(b\in\Z)$, $v_4=(0,1)$ is associated to an origami template only when $b=\pm1$ or $\pm2$.

(2) Orlik-Raymond (\cite{or-ra70}) further prove that a simply connected compact smooth 4-manifold (except $S^4$) with an effective smooth action of $T^2$ decomposes into an equivariant connected sum of copies of $\C P^2$, $\overline{\C P^2}$ and Hirzebruch surfaces with natural $T^2$-actions.  Since an equivariant connected sum of toric origami manifolds is again a toric origami manifold, Theorem~\ref{theo:5.1} can be proved using the decomposition result and our proof in Proposition~\ref{prop:5.1} can be regarded as another proof to the decomposition result.
\end{remark}

A. Cannas da Silva \cite{cann10} shows that any orientable compact smooth (not necessarily simply connected) $4$-manifold admits a folded symplectic form.  However, the toric origami version of this statement is not true.  For example, the 4-dimensional torus with a free $T^2$-action does not admit a toric origami form because its orbit space is the $2$-dimensional torus and not homotopy equivalent to a bouquet of $S^1$.

\section{Products of toric origami manifolds} \label{sec:6}

Let $(M,\omega)$ be a toric origami manifold with nonempty fold $Z$.  If $(M',\omega')$ is a symplectic toric manifold, then the product $(M\times M',\omega\times\omega')$ is again a toric origami manifold (with fold $Z\times M'$).  However, if $(M',\omega')$ is an origami manifold with nonempty fold $Z'$, then $\omega\times\omega'$ is not an origami form on $M\times M'$ because $\omega\times \omega'$ does not have a maximal rank at a point in $Z\times Z'$.  So, it is natural to ask the following.

\medskip
\noindent
{\bf Problem.}  Let $M$ and $M'$ be toric origami manifolds.  If $M\times M'$ with the product action admits a toric origami form, then does either $M$ or $M'$ admit a toric symplectic form?

\begin{remark}
A toric origami form on $M\times M'$ is not necessarily the product of a toric origami form and a toric symplectic form.  For example, let $P$ be a right-angled trapezoid and let $F$ be the side of $P$ which does not touch the two right-angled corners.  Let $\bar{P}$ and $\bar F$ be copies of $P$ and $F$ respectively.  Then $\cP=\{P,\bar P\}$ and $\cF=\{\{F,\bar F\}\}$ form an origami template $(\cP,\cF)$ and the toric origami manifold associated to $(\cP,\cF)$ is equivariantly diffeomorphic to $S^2\times S^2$ with the standard product action of $T^2$.  But the toric origami form associated to $(\cP,\cF)$ is not a product because $P$ is not a product of two simple polytopes.
\end{remark}

In this section we make some observations on the problem above, all of which provide affirmative evidence to the problem.

\begin{lemma} \label{lemm:6.1}
Let $S^4$ be the standard $4$-sphere with the standard $T^2$-action.  Then $S^4\times S^4$ with the product $T^4$-action does not admit a toric origami form.
\end{lemma}

\begin{proof}
Suppose that $S^4\times S^4$ with the product action admits a toric origami form and let $(\cP,\cF)$ be the associated origami template.  Since $S^4\times S^4$ is simply connected, $(\cP,\cF)$ is acyclic (i.e., the graph $\Gamma(\cP,\cF)$ is a tree) by Proposition~\ref{prop:3.1}.  The polytope corresponding to a leaf in the tree $\Gamma(\cP,\cF)$ must have a vertex not contained in any folding facet.  On the other hand, the topological space $|(\cP,\cF)|$ is homeomorphic to the orbit space $(S^4\times S^4)/T^4$ as manifolds with corners as remarked after Example~\ref{exam:3.1}, and $S^4\times S^4$ with the $T^4$-action has only four characteristic submanifolds and the characteristic submanifolds correspond to the facets of the orbit space $(S^4\times S^4)/T^4$.  Therefore, $|(\cP,\cF)|$ has only four facets and hence a polytope corresponding to a leaf in $\Gamma(\cP,\cF)$ must be a $4$-simplex and one of the facets of the polytope must be a folding facet.

Let $P_1$ be a polytope in $\cP$ corresponding to a leaf of $\Gamma(\cP,\cF)$.  As noted above, $P_1$ is a $4$-simplex and has one folding facet.  Since the number of facets in $|(\cP,\cF)|$ is four and $P_1$ has already four facets except the folding facet, the polytope $P_2$ which shares the folding facet of $P_1$ must be combinatorially equivalent to the product of a $3$-simplex and a $1$-simplex.  Again since the number of facets in $|(\cP,\cF)|$ is four, the facet of $P_2$ which does not meet the folding facet of $P_1$ must also be a folding facet.  Repeating this argument shows that we finally reach another $4$-simplex and this means that the tree $|(\cP,\cF)|$ has only two leaves and there are only two vertices in $|(\cP,\cF)|$.  However, $|(\cP,\cF)|$ must have four vertices because the $T^4$-action on $S^4\times S^4$ has four fixed points.  This is a contradiction and the lemma is proven.
\end{proof}

\begin{theorem} \label{theo:6.1}
Let $S^{2n_i}$ $(n_i\ge 1)$ be the standard $2n_i$-sphere with the standard $T^{n_i}$-action for $i=1,\dots,k$.  Then $\prod_{i=1}^k S^{2n_i}$ with the product action of $\prod_{i=1}^k T^{n_i}$ admits a toric origami form if and only if $n_i=1$ except for one $i$.
\end{theorem}

\begin{proof}
Any characteristic submanifold of a toric origami manifold is again a toric origami manifold with the restricted form.  Therefore, any connected component of intersections of the characteristic submanifolds is a toric origami manifold.  If $\prod_{i=1}^k S^{2n_i}$ with the product action of $\prod_{i=1}^k T^{n_i}$ admits a toric origami form, then so is any connected component of $\prod_{i=1}^k S^{2m_i}$ with the standard action of $\prod_{i=1}^k T^{m_i}$ for any $0\le m_i\le n_i$.  This observation together with Lemma~\ref{lemm:6.1} proves the \lq\lq only if part" of the theorem.

Any $S^{2n}$ with the standard $T^n$-action admits a toric origami form (see Example~\ref{exam:3.1}) and $S^2$ with the standard $S^1$-action admits an invariant symplectic form since $S^2=\C P^1$.  The product of these forms is a toric origami form, so the \lq\lq if part" of the theorem follows.
\end{proof}

Here is another example which provides an affirmative evidence to the problem above.

    \begin{proposition}
        A manifold $S^{2n}\times T^2$ $(n\ge 2)$ with a standard product torus action does not admit a toric origami form.
    \end{proposition}

    \begin{proof}
        It suffices to prove the lemma when $n=2$ by the same reason in the proof of Theorem~\ref{theo:6.1}.
       Suppose that $S^{4}\times T^2$ with the product action of $T^3=T^2\times T^1$ admits a toric origami form and let $(\cP,\cF)$ be the associated origami template. Since $(S^{4}\times T^2)/T^{3}$ is homeomorphic to a solid torus $D^2\times S^1$, we can find two sub-origami templates $(\cP_1,\cF_1)$ and $(\cP_2,\cF_2)$ such that
        \begin{enumerate}
            \item $(\cP_1,\cF_1)$ and $(\cP_2,\cF_2)$ are acyclic;
            \item $\cP=\cP_1\cup \cP_2$;
            \item $\cF=\cF_1\cup\cF_2\cup\{(F_1,G_1),(F_2,G_2)\}$, where $F_i$ is a facet of a polytope in $\cP_1$ and $G_i$ is a facet of a polytope in $\cP_2$ for $i=1,2$.
        \end{enumerate}
        Since $|(\cP,\cF)|$ has no vertex, $|(\cP_1,\cF_1)|$ and $|(\cP_2,\cF_2)|$ are homeomorphic to $F_1\times\Delta^1$ as manifolds with corners. Since each $F_i$ and each $G_i$ are facets of some Delzant polytopes, they are also Delzant polytopes. Let $M_{F_1}$ be the symplectic toric manifold corresponding to $F_1$. Then the toric origami manifolds corresponding to $(\cP_1,\cF_1)$ and $(\cP_2,\cF_2)$ are equivariantly diffeomorphic to $M_{F_1}\times S^2$. Therefore, the toric origami manifold corresponding to $(\cP,\cF)$ is $M_{F_1}\times T^2$. This is a contradiction because $H_2(S^{4}\times T^2)\not\cong H_2(M_{F_1}\times T^2)$.
    \end{proof}

We conclude this paper with two more observations on the problem mentioned in this section.

\begin{lemma}
Let $M$ and $M'$ be co\"orientable toric origami manifolds. If $M\times M'$ with the product action admits a toric origami form, then either $M$ or $M'$ is simply connected.
\end{lemma}

\begin{proof}
Suppose that both $M$ and $M'$ are not simply connected.  Then their orbit spaces $M/T$ and $M'/T'$ are also not simply connected by Proposition~\ref{prop:3.1}.  Since $M$ and $M'$ are toric origami manifolds, both $M/T$ and $M'/T'$ are homotopy equivalent to (non-trivial) bouquets of $S^1$.  Therefore the product $M/T\times M'/T'$ has a non-trivial homology in degree $2$ and hence $M\times M'$ with the product action does not admit a toric origami form.
\end{proof}

Remember that a facet in an origami template $(\cP,\cF)$ is called non-folded if it does not touch any facet in $\cF$.

\begin{lemma} \label{lemm:6.3}
Let $M$ and $M'$ be co\"orientable toric origami manifolds. If $M\times M'$ with the product action admits a toric origami form whose associated origami template has a non-folded facet, then either $M$ or $M'$ admits a toric symplectic form which may be different from the original toric origami form on $M$ or $M'$.
\end{lemma}

\begin{proof}
Let $(\cP,\cF)$ be an origami template associated to $M\times M'$.  The topological space $|(\cP,\cF)|$ is a manifold with corners and a codimension-$1$ face of $|(\cP,\cF)|$ corresponds to a codimension-$2$ closed connected submanifold of $M\times M'$ fixed pointwise under some circle subgroup of $T\times T'$.  Such a circle subgroup is contained in either $T$ or $T'$ because the $T$-action on $M$ and $T'$-action on $M'$ are both locally standard.

Let $F$ be a non-folded facet in $(\cP,\cF)$ and $S$ be a circle subgroup which fixes a codimension-$2$ closed connected submanifold $Q$ of $M\times M'$ corresponding to $F$.  Then $Q$ admits a toric symplectic form and $Q$ is a connected component of the $S$-fixed point set $(M\times M')^S$.  Note that $S$ is contained in either $T$ or $T'$ and we may assume that $S$ is contained in $T$ without loss of generality. Then $(M\times M')^S=M^S\times M'$.
Since $Q$ is a connected component of $(M\times M')^S=M^S\times M'$, $Q$ is of the form $Q_1\times M'$ where $Q_1$ is a connected component of $M^S$.  Since $Q=Q_1\times M'$ admits a toric symplectic form,  so is any connected component of the $T$-fixed point set $Q^T=Q_1^T\times M'\subset M^T\times M'$.  Here, $Q_1^T(\subset M^T)$ consists of finitely many points, so any connected component of $Q^T$ can be identified with $M'$. Therefore $M'$ admits a toric symplectic form.
\end{proof}

The assumption in Lemma~\ref{lemm:6.3} seems satisfied in most cases. But the following observation shows that there is no topological condition on a toric origami manifold which ensures the existence of a non-folded facet in the associated origami template.  Let $M$ be a toric origami manifold of dimension $2n$ and $(\cP,\cF)$ be the associated origami template.  An equivariant connected sum of $M$ with the toric origami manifold $S^{2n}$ in Example~\ref{exam:3.1} (with an action of $T$ through an appropriate automorphism of $T$) produces a toric origami manifold $M'$ equivariantly diffeomorphic to $M$. But if $v$ is the vertex in $(\cP,\cF)$ corresponding to the fixed point in $M$ used to perform the equivariant connected sum and $F$ is a facet in $(\cP,\cF)$ containing $v$, then the corresponding facet $F'$ to $F$ in the origami template $(\cP',\cF')$ associated to $M'$ is not non-folded.  Therefore, performing this equivariant connected sum operation at all fixed points in $M$, one can produce a toric origami manifold which is equivariantly diffeomorphic to $M$ but whose associated origami template has no non-folded facet.

\end{document}